\documentclass[10pt]{amsart} 

\newtheorem{defi}{\bf D\scriptsize EFINITION \normalsize}
\newtheorem{theorem}{\bf T\scriptsize HEOREM \normalsize}
\newtheorem{lm}{\bf L\scriptsize EMMA \normalsize}
\newtheorem{dk}{\bf C\scriptsize OROLLARY \normalsize}
\newtheorem{rem}{\bf R\scriptsize EMARK \normalsize}
\newtheorem{exa}{\bf E\scriptsize XAMPLE \normalsize}
\newtheorem{pro}{\bf P\scriptsize ROBLEM \normalsize}
\newtheorem{prop}{\bf P\scriptsize ROPOSITION \normalsize}
\newtheorem{no}{\bf N\scriptsize OTE \normalsize}

\def\kopr{\hfill\raisebox{3pt}{\framebox{$\star$}}}
\newenvironment{example}{\begin{exa}\rm}{$\kopr$\end{exa}}
\newenvironment{definition}{\begin{defi}\rm}{\end{defi}}

\newenvironment{lemma}{\begin{lm}\it}{\end{lm}}

\begin{document}

\title{Local distributional chaos}
\author{Francisco Balibrea}
\author{Lenka Ruck\'a }

{\hskip 115mm  November 25, 2021 
\vskip 5mm}

\address{Department of Mathematics, University of Murcia, 30100  Murcia, Spain}
\address{Mathematical Institute, Silesian University in Opava, 74601 Opava, Czechia}
\address{~}

\email{balibrea@um.es}
\email{lenka.rucka@math.slu.cz}

\maketitle

\pagestyle{myheadings}
\markboth{{\sc F. Balibrea, L. Ruck\'a}}
{{\sc Local distributional chaos}}

\begin{abstract} 

Given the dynamical system $(I, f)$ where $I = [0,1]$ and $f \in C(I,I)$, in \cite{14} was introduced the notion of \textit{distributional chaos} which tries to understand the way in which some different types of chaos are spread out on $I$.
\medskip
 
In general discrete dynamical systems $(X, f)$ where $X$ is a topological space and $f \in C(X,X)$, three notions of distributional chaos were defined and denoted by $DC1, DC2$ and $DC3$, were introduced in \cite{6} to distinguish three different ways of distributing. In $(I, f)$ such three notions coincide and they will be denoted by DC-chaos. Generally speaking we have $DC1 \subseteq DC2 \subseteq DC3$--chaos.

\medskip

We wonder if it is possible that chaos can concentrate in some points and develop a local idea of the distributional chaos. Answering to this question, in \cite{1} is introduced the new notion of $DCi$--points for $i = 1,2,3$. Such special points are those in which DC--chaos of different types concentrate.

\medskip
Also in \cite{1} it is proved that if $f$ is continuous interval map with positive topological entropy, then there is at least one DC1--point in the system.

\medskip

In this paper it is proved that in the symbolic space $(\Sigma, \sigma) $ where $\sigma$ is the shift map, every point of $\Sigma$ is a DC1--point. This result is necessary to prove one of the main results of the paper. If $h(f) > 0 $ for an interval system, then it has an uncountable set of DC1--points and moreover the set can be chosen perfect.

\medskip

In greater dimensions than one, we deal with \textit{triangular systems} on $I^{2}$. In this case the relationship between topological entropy and different cases of distributional chaos is not clearly understood and several different results are possible. In the paper we use an example of $F$ given by Kolyada in \cite{10} to prove that the corresponding two dimensional system $(I^ {2}, F)$ has positive topological entropy but without containing DC2--points, proving that there is no concentration of DC2--chaos

\end{abstract}

\section{Introduction}
\medskip

In 2019, A. Loranty and R. Pawlak in \cite{1} dealt with the problem if the distributional chaos \cite{14}, denoted by $DCi$ where $i = 1,2,3$, can concentrate around some points or not, following a \textit{local distributional chaos} theory. In positive cases, such points in some sense, attract the generate chaos. One of the main results on \cite{1} states that if $f$ is a continuous interval map with positive topological entropy h(f) $> 0$, then there is at least one $DC1-point$.   \\

In this paper in Section \ref{result 1}, we give a stronger global version of such result and prove that  if $h(f) > 0$, then there is an uncountable number of $DC1-points$ and the set of them is a Cantor set.   \\

In section \ref{result 2} we concern with the notion of $DCi-points$ in more general spaces, in particular, in the class of triangular maps of the square as a system of higher dimension than interval, but holding many properties similar to those valid in the interval case. For example almost all Sharkovsky's classification is known not only for interval maps, but also for triangular maps, see \cite{8}. \\

In case of triangular maps we need to distinguish between different versions of distributional chaos, DC1, DC2 and DC3 (see e.g. \cite{6}). Main result here is that for triangular maps the strict version of Loranty-Pawlak result is not true for none of versions of DCi-points. For example the construction in \cite{7} and \cite{10} gives a triangular map with positive topological entropy, with no DC1--scrambled pair, but with an uncountable DC2-scrambled set. We show that despite this, there is no DC1, DC2, or DC3 point in the construction.

\section{Preliminaries}
\smallskip

In this section we recall notions we are using. Throughout the paper we use standard notation. Let $(X,d)$ be compact metric space with metric $d$ and $f: X \to X$ a continuous map. Then $\text{diam} (A) = \max\{d(x,y); x,y \in A\}$ for $A \subseteq X$. We use $f^n$ to denote $n$-th iteration of $f$, which is $f^n = f ^{n-1} \circ f$, where $n \in \mathbb{N}$ (and $f^0$ is the identity map). \\

The notion of topological entropy we are using is that introduced by Adler, Konheim and McAndrew (see \cite{11}) defined for continuous maps on topological space $(X, \tau)$, where $\tau$ is its suitable topology.s

We say that a system $(X,f)$ possesses a horseshoe (or $n$-horseshoe), if there are nonempty and nondegenerate sets $J_1, ..., J_n$ with pairwise disjoint interiors such that $(J_1 \cup ... \cup J_n) \subset f(J_i)$ for all $i=1, ..., n$. If sets $J_1, ... ,J_n$ are disjoint, we call it a strict horseshoe. \\

Further we proceed to repeating the definition of distributional chaos, see e.g. \cite{6}. Given a metric $d$, for any positive integer $n$, real number $t$ and pair of points $x,y$, we denote by $\xi(x,y,t,n)$ the number of members of the set $\{d(x,y), d(f(x), f(y)), d(f^2(x), f^2(y)), ...,  d(f^{n-1}(x), f^{n-1}(y))\}$ which are less than $t$. Clearly $\xi(x,y,t,n)=0$ for $t \leq 0$, while for $t>0$, $\xi(x,y,t,n) \in \{0, 1, ..., n\}$. Then we introduce the function $\Psi_{xy}^n(f,t)$ associate to $f$ as $\Psi_{xy}^n(f,t) = \frac{1}{n} ~\xi(x,y,t,n)$ for map $f$. Function $\Psi_{xy}^n(f,t)$ is non-decreasing and its values are from interval [0,1]. Next, let 
\begin{equation}
\Psi_{xy}(f,t) = \liminf_{n \to \infty} ~\Psi_{xy}^n(f,t) \hspace{0.4cm} \text{and} \hspace{0.4cm} \Psi^*_{xy}(f,t) = \limsup_{n \to \infty} ~\Psi_{xy}^n(f,t)
\label{DC1}
\end{equation}
be lower and upper distributional function of $f$ with respect to points $x,y$. It is easy to see that $\Psi_{xy}(f,t) \leq \Psi^*_{xy}(f,t)$ for all $t \in \mathbb{R}$. \\

We say that pair of points $x,y$ form a 
\begin{itemize}
\item DC1--scrambled pair of $f$ (or just DC1 pair), if $\Psi^*_{xy}(f,t) \equiv 1$ and $\Psi_{xy}(f,t)=0$ for some $t>0$,
\item DC2--scrambled pair of $f$ (or just DC2 pair), if $\Psi^*_{xy}(f,t) \equiv 1$ and $\Psi_{xy}(f,t)<\Psi^*_{xy}(f,t)$ ~~ $\forall t>0$,
\item DC3--scrambled pair of $f$ (or just DC3 pair), if $\Psi_{xy}(f,t)<\Psi^*_{xy}(f,t)$ for all $t>0$.
\end{itemize}

Map $f$ is called DCi-scrambled (or simply DCi), if there exists an uncountable subset $S \subseteq X$, where all pairs of points are DCi-scrambled pairs. \\

Clearly, DC3 chaos is weaker than DC2 and DC2 chaos is weaker than DC1. It is known, for example, that positive topological entropy in general dynamical system implies DC2, but not DC1. In the class of continuous maps of the interval, DC1, DC2 and DC3 are proven to be equivalent. There are also other versions of distributional chaos appearing in the theory of linear operators (see \cite{17}). \\

Next definition is an extension of the one taken from \cite{1} and introduces the notion of $DCi-points$.

\begin{definition}
(\cite{1}) Let $(X,f)$ be a compact metric space. Point $x_0$ is called a DCi--point of $f$, for $i= 1,2,3$,  if for every $\varepsilon>0$ there are an uncountable set $S_\varepsilon$, integer $n(\varepsilon)$ and a closed set $A_\varepsilon \supset S_\varepsilon$, such that $S_\varepsilon$ forms a DCi--scrambled set for $f$ and $A_\varepsilon \subset f^{n(\varepsilon)}(A_\varepsilon) \subset B(x_0, \varepsilon)$.
\label{DCpoint}
\end{definition}

The set $A_\varepsilon$ is called \textit{an envelope} of the system $(X,f)$ where the neighbourhood $B(x_0, \varepsilon)$ depends on the chosen metric in $X$. In the case of interval map we will use common term DC--point for all three equivalent versions DC1, DC2, DC3.  \\

A triangular map (sometimes also called skew-product maps usually when we are describing situations on spaces which are not of coordinates) on the square $[0, 1]^{2} = I^2$ into itself is a continuous map $F: I^2 \to I^2$ of the form $F(x,y) = (f(x), g(x,y)) = (f(x), g_x(y))$, where $I=[0,1]$ and $g_x: I \to I$ is a family of maps depending continuously on $x$. The map $f$ of the corresponding dynamical system is called the base for $F$ and the maps $g_x$ are called the fibre maps for all $x \in I$. \\

\section{DC-points for the interval maps}
\label{result 1}
\medskip
The goal of this section is to strengthen the result from \cite{1} that every positive entropy interval map posses at least one DC--point. Here we show that in this case there is not only one DC--point, but uncountably many of them and moreover, their set can be chosen perfect (Cantor). We start with recalling that symbolic system is DC1 and presenting the Lemma saying that all points in symbolic space are DCi--points. It suggests that these points are not so special, as their definition might imply.\\

Let $\Sigma = \{0,1\}^{\mathbb{N}}$ denote the space of all infinite sequences of 0's and 1's and $\sigma \in C(\Sigma)$ be the one-sided shift map defined as $\sigma(\alpha_1 \alpha_2\alpha_3 ...) = (\alpha_2 \alpha_3 ... )$ where $(\alpha_1 \alpha_2 ... )\in \Sigma$. It is known that system $(\Sigma, \sigma)$ contains (at least one) uncountable DC1--scrambled set. Its construction can be found for example in \cite{5} (in the proof of Theorem 6.26). Let's denote such set by $S \subset \Sigma$. \\

\begin{lemma}
Every point of symbolic system $(\Sigma, \sigma)$ is a $DC1-point$. 
\label{sigma}
\end{lemma}

\begin{proof}
Equip the space $\Sigma$ with the metric $d$, where $d(x,y) = \frac1i$ and $i\in \mathbb{N}$ is the first position where symbols in $x,y$ differ. Given $B(x_0, \varepsilon)$, let $m(\varepsilon)$ denote the lowest integer such that for any point $x \in B(x_0, \varepsilon)$, sequences $x$ and $x_0$ coincide on (at least) first $m(\varepsilon)$ positions. \\

We choose point $x_0 \in \Sigma$ and fix it. If we want to show that $x_0$ is a DC1--point, by definition we need to show that for every $\varepsilon>0$ there are: uncountable DC1--scrambled set $S_{x_0,\varepsilon}$, integer $n(x_0,\varepsilon)$ and a closed set $A_{x_0,\varepsilon} \supset S_{x_0,\varepsilon}$ such that $A_{x_0,\varepsilon} \subset \sigma^{n(x_0,\varepsilon)}(A_{x_0,\varepsilon}) \subset B(x_0, \varepsilon)$. \\

We use notation $x_0|_{n(x_0,\varepsilon)}$ for finite block of first $n(x_0,\varepsilon)$ symbols of sequence $x_0$. For any $\varepsilon>0$ we set $n(x_0,\varepsilon):=m(\varepsilon)$, $S_{x_0,\varepsilon}:=\{x_0|_{n(x_0,\varepsilon)} ~x_0|_{n(x_0,\varepsilon)} ~y, ~\forall y \in S\}$ and $A_{x_0,\varepsilon}:=\{x_0|_{n(x_0,\varepsilon)} ~x_0|_{n(x_0,\varepsilon)} ~x, ~\forall x \in \Sigma\}$. This makes $\sigma^{n(x_0,\varepsilon)} (A_{x_0,\varepsilon}) =\{x_0|_{n(x_0,\varepsilon)} ~x, ~\forall x \in \Sigma\}$.\\

Obviously $S_{x_0,\varepsilon} \subset A_{x_0,\varepsilon} \subset \sigma^{n(x_0,\varepsilon)} (A_{x_0,\varepsilon}) = B(x_0, \varepsilon)$. Moreover $S_{x_0,\varepsilon}$ forms an uncountable DC1-scrambled set of $\Sigma$, because $S$ is an uncountable DC1--scrambled set of $\Sigma$ and $A_{x_0,\varepsilon}$ is closed set, because limit point of sequence of points of type $~x_0|_{n(x_0,\varepsilon)} ~x_0|_{n(x_0\varepsilon)} ~x~$ is again point of type $~x_0|_{n(x_0,\varepsilon)} ~x_0|_{n(x_0\varepsilon)} ~x$.
\end{proof}
\medskip

\begin{theorem}
Let $f:[0,1]\to[0,1]$ be a continuous map such that $h(f) >0$. Then the system ([0,1],f) has an uncountable set of DC1-points. Moreover this set can be chosen perfect.
\label{th1}
\end{theorem}
\smallskip

\begin{proof}
We start with the Misiurewicz's result that the positive topological entropy of $f$ implies existence of a horseshoe for some iteration of $f$ (e.g. Theorem 4.7 in \cite{5}). Let us denote such iteration by $f^k$ and for simplicity $g:=f^k$. From now on we work with map $g$ and we will get back to $f$ in the end of the proof. \\

Map $g$ possesses a horseshoe, so we get sets $J_0, J_1$ such that $(g(J_0) \cap g(J_1)) \supset (J_0\cup J_1)$. By applying Lemma 1.13(i) from \cite{5} we get subsets $J_{00}, J_{01} \subset J_0$ such that $g(J_{00}) = J_0$ and $g(J_{01}) = J_1$. After doing the same for $J_1$ we get sets $J_{10}, J_{11}$. Also by Lemma 1.13(i) it is true that sets $J_{00}, J_{01}, J_{10}, J_{11}$ are disjoint, their boundaries map on boundaries of $J_0$ or $J_1$ and their interiors on interiors of $J_0$ or $J_1$, respectively. By additional applications of Lemma 1.13(i) and analysis similar to that from Proposition 5.15 from \cite{5}, we get system of sets $J_{\overline{\alpha}} = \bigcap_{n=1}^{\infty} J_{\alpha_1 ... \alpha_n}$, where $\overline{\alpha} = \alpha_1 \alpha_2 .... \in \Sigma$. \\

Also, by Proposition 5.15 from \cite{5}, there exists an invariant Cantor set $X \subset I$ and a continuous map $\varphi: X \to \Sigma $ such that $\varphi$ is a semi-conjugacy between $(X,g|_X)$ and $(\Sigma, \sigma)$ (in particular $(\varphi \circ g)(x) = (\sigma \circ \varphi)(x)$ for all $x \in X$), the system $(X,g|_X)$ is transitive and there exists countable set $E \subset X$ such that $\varphi$ is one-to-one on $X \setminus E$ and two-to-one on $E$.\\

By Lemma \ref{sigma}, every point of $(\Sigma, \sigma)$ is a DC--point. As in its proof, we denote the uncountable DC1-scrambled set in $(\Sigma, \sigma)$ by $S$ and uncountable DC1--scrambled set with respect to point $\overline{\alpha}$ and $\varepsilon>0$ as $S_{\overline{\alpha}, \varepsilon}$. Given $\overline{\alpha} \in \Sigma$ and $\varepsilon>0$ define the transformation 
 $\psi_{\overline{\alpha},\varepsilon}: \Sigma \to \Sigma$ as
 
\begin{equation}
\psi_{\overline{\alpha},\varepsilon}(x) = \overline{\alpha}|_{n(\overline{\alpha},\varepsilon)}~ \overline{\alpha}|_{n(\overline{\alpha},\varepsilon)}~ x~~\text{for all} ~~ x\in \Sigma.
\label{BB}
\end{equation}

The transformation maps all space $\Sigma$ to $A_{\overline{\alpha},\varepsilon}$ and DC1--scrambled set $S$ into $S_{\overline{\alpha},\varepsilon}$. Moreover, map $\psi_{\overline{\alpha},\varepsilon}$ is continuous for every $\overline{\alpha} \in \Sigma$ and $\varepsilon>0$, because for arbitrary points $x\neq y \in \Sigma$, the images $\psi_{\overline{\alpha},\varepsilon}(x), \psi_{\overline{\alpha},\varepsilon}(y)$ have $2n(\overline{\alpha},\varepsilon)$ more common first spaces than $x,y$. \\

As $\varphi$ maps one point from $X$ to one point from $\Sigma$, $\psi_{\overline{\alpha},\varepsilon}$ maps one point in $\Sigma$ to one point in $\psi_{\overline{\alpha}, \varepsilon}(\Sigma)$. Let us suppose the point $x_0= \varphi^{-1}(\overline{\alpha})$. By the construction of $X$, $\psi_{\overline{\alpha},\varepsilon}(\Sigma)$ and map $\varphi$, $x_0$ is a DC1--point of $g$ with uncountable DC1--scrambled set $S_{x_0,\varepsilon} = \varphi^{-1} (\psi_{\overline{\alpha}, \varepsilon}(S))$ and envelope $A_{x_0,\varepsilon} =\varphi^{-1} (\psi_{\overline{\alpha},\varepsilon}(\Sigma))$ in the referred topology on $\Sigma$. Consequently the set $X=\varphi^{-1}(\Sigma)$ is an uncountable set of DC1--points of $g$ and by the above construction it is a Cantor set. \\

As a last step we want to show that $X$ is also uncountable set of DC1--points for $f$. We fix an arbitrary pair $x,y \in X$ and $t \in I$. Since for interval maps all versions of distributional chaos are equivalent, it is enough to use chaos DC3 in the proof. We know $\Psi_{xy}^*(g,t) >\lambda > \Psi_{xy}(g,t)$ for some $\lambda \in (0,1)$ and we want to show the same for map $f$. Let us suppose, by contraction, that there is $\gamma \in I$ such that $\Psi_{xy}^*(f,t) = \Psi_{xy}(f,t) =\gamma$. Then $\lim_{n \to \infty} \Psi_{xy}^n(f,t) = \gamma$ and consequently for every $\delta>0$ there is an $N \in \mathbb{N}$ such that for all $n>N$, $\Psi_{xy}^n(f,t) \in (\gamma-\delta, \gamma + \delta)$. Map $f$ is continuous, thus for every $\varepsilon>0$ there exists $\delta>0$ such that 

\begin{equation}
\frac{\xi(x,y,t,n)}{n} \in (\gamma-\delta, \gamma+\delta) ~~{\rm for}~f ~~~~\Longrightarrow~~~~ \frac{\xi(x,y,t,\lfloor \frac{n}{k} \rfloor)}{\lfloor \frac{n}{k} \rfloor} \in (\gamma-\varepsilon, \gamma+\varepsilon)  ~~{\rm for}~g,
\label{new}
\end{equation}

where $\lfloor \cdot \rfloor$ denotes the integer part of the rational number. But (\ref{new}) implies $\lim_{n \to \infty} \Psi_{xy}^n(g,t)=\gamma =\lambda$, which finishes the contradiction. \\

Finally since $x,y$ and $t$ was arbitrary, we showed that $X$ is uncountable set of DC1--points of map $f$, as desired. 
\end{proof}
\bigskip

\section{DC--points in more general spaces}
\label{result 2}
\medskip
In this section we study existence of DCi-points in higher classes of maps, general compact metric spaces with continuous maps. There is no longer equivalence between various types of distributional chaos. \\

Generalization of the result from section \ref{result 1} for DC1--points is not possible here without supposing some additional properties, because there is generally no relation between positive topological entropy and chaos DC1, see \cite{15} and \cite{16}. On the other hand, Downarowicz proved in \cite{9} that positive topological entropy implies chaos DC2 (hence also DC3). Despite this result, positive topological entropy does not imply existence of a single DC2--point, even in the class of triangular maps of the square. In the paper by Kolyada \cite{10}, there appears a construction of a whole class of continuous triangular maps of type $2^\infty$ with positive topological entropy. Later, it was showed by Smítal and Štefánková in \cite{7} that these maps have no DC1--scrambled pair, but there is uncountable DC2--scrambled set. We use this construction to show that there is no DC2--point in any of the maps. \\

By the definition \ref{DCpoint}, the envelope $A_{\varepsilon}$ of a DCi--point $x_0$ must satisfy the inclusions $A_\varepsilon \subset f^{n(\varepsilon)}(A_\varepsilon) \subset B(x_0, \varepsilon)$. We will show that in the construction by Kolyada, no possible envelope can satisfy the first inclusion, while it is possible to satisfy the second one. It could be discussed whether the first inclusion is necessary in the definition, or how it can be replaced. But in this paper we omit such discussion and we take the definition as presented in \cite{1}.\\

First we recall the Kolyada's construction and then, with help of Lemma \ref{helpful} we prove Theorem \ref{noDCpoint}, the main result of this section. 

\begin{example}
We recall the construction of class of triangular maps $F(x,y) = (f(x), g_x(y))$ from Theorem 10 in \cite{10}, for more details see the cited paper. \\

Let base map $f$ be the logistic function $f(x) = \lambda x (1-x)$, where $\lambda$ is irrational number $\lambda=3.569...$, such that $f$ is of type $2^\infty$. Such map is known to contain a system of intervals $J_k^n$, where $n=1, 2, ...$and $k=0, 1, ..., 2^n-1$ such that:
\begin{equation}
\text{The} \hspace{0.2cm}J_k^n ~~\text{with fixed}~ n~ \text{are mutually disjoint,}
\label{form1}
\end{equation}
\begin{equation}
f(J_k^n)=J_{k+1}^n ~~\text{and}~~ f(J_{2^n-1}^n) = J_0^n,
\label{form2}
\end{equation}
\begin{equation}
\text{every interval} ~~J_k^{n-1}~~ \text{contains exactly two intervals} ~~J_k^n, J_{k+2^{n-1}}^n ~\text{and} 
\label{form3}
\end{equation}
$$\text{its complement}~~ J_k^{n-1}\setminus (J_k^n \cup J_{k+2^{n-1}}^n)~~ \text{contains periodic point of type} ~~2^{n-1},$$
\begin{equation}
\lim_{n \to \infty} \max_{0 \leq k \leq 2^n-1} |J_k^n| = 0~~ \text{and} ~~Q = \bigcap_{n=0}^\infty \bigcup_{k=0}^{2^n-1} J_k^n~~ \text{is a minimal Cantor set.} 
\label{form4}
\end{equation}
\smallskip

Fibre maps $g_x(y) \equiv 0$ for $x=0,\frac12,1$ and in all periodic points of $f$. On set $Q$, $g_x$ are tent maps divided by different powers of 2. In particular if $\tau (y) =1-|1-2y|$ is a tent map, we set 

\begin{equation}
g_x(y) = \frac{\tau(y)}{2^{i-1}} ~~~\text{for}~~~ x \in Q \cap J^n_{2^{n-1}},
\label{map g}
\end{equation}
 where $n \in [n_i, n_{i+1}-1]$ and sequence $\{n_i\}_{i=1}^\infty$ is to be defined in the next paragraph. The remaining fibre maps are filled in also by tent maps divided by suitable real numbers chosen in such a way that heights of $g_x$ increase or decrease linearly with respect to $x$, which secure the continuity of  $F$.\\

Sequence $\{n_i\}_{i=1}^{\infty}$ must satisfy the conditions $n_1= 1$ and
\begin{equation}
\sum_{i=1}^{\infty} \frac{1}{2^{n_i}}<1, 
\label{condition}
\end{equation}
so it grows faster than sequence of integers. With various choice of sequences $\{n_i\}$ we get whole class of triangular maps, which is denoted by $\mathcal{F}$.
With given conditions for $\{n_i\}$ satisfied, every map $F \in \mathcal{F}$ has positive topological entropy (see \cite{10}, \cite{7}) and there is no DC1--scrambled pair (\cite{7}). On the other hand there exists an uncountable DC2--scrambled set in each fibre above point of $Q$. Originally, in \cite{7} is shown the existence of one DC2--scrambled pair, but by Downarowicz's result (\cite{9}), there can be found a whole uncountable DC2--scrambled set, which can be chosen as a Cantor set.
\label{example}
\end{example}

\bigskip

For any $x\in I$ and any $F\in \mathcal F$,  let $R(F,x,j)$ be the range of  $F^j$ restricted to the fibre $I_x$. Also let denote by $K_0 = \{x \in I; g_x=0\}$ and by $K$ the union of all both-sided $f-$orbits of points from $K_0$. Set $K$ is a countable subset of $I$.  

\medskip

\begin{lemma}
Given $\{n_k\}$ and $F \in \mathcal{F}$ as in the previous example, for any $J_{k_0}^n$ and any point $x \in ((Q \cap J_{k_0}^n )\setminus K)$:
\begin{enumerate}
\item $\limsup_{j\to\infty} R(F,x,j)=1,$
\item $\liminf_{j\to\infty}R(F,x,j)=0.$
\end{enumerate}
\label{helpful}
\end{lemma}

\begin{proof}
As was mentioned, the sequence $\{n_i\}_{i=1}^\infty$ is bounded from below by sequence of integers. Let's now suppose the case of triangular map $F (\notin \mathcal{F})$ where $n_i=i$ and point $(x_0,0)$ such that $x_0 \in Q \setminus K$ and $g_{x_0}$ is full tent map (such $x_0$ exists, because $n_1=1$). By (\ref{map g}),
\begin{equation}
g_{f^{i-1+n 2^i}(x_0)} = \frac{\tau}{2^{i-1}}, i=1,2,...;~ n=0,1,2,...~ .
\end{equation}\\
It means, starting at $x_0$, we apply fibre maps in order $\tau, \frac{\tau}{2}, \tau,  \frac{\tau}{4}, \tau,  \frac{\tau}{2}, \tau,  \frac{\tau}{8}, \tau,  \frac{\tau}{2}, \tau,  \frac{\tau}{4}, \tau,  \frac{\tau}{2}, \tau, \frac{\tau}{16}, \tau, ...$. Taking small set $\{x_0\} \times [0,\frac{1}{2^j}] $ for some $j$, every full tent map spreads its length 2 times, while every application of $\frac{\tau}{2^i}$ shrinks the vertical length $2^{i-1}$ times. It is easy to count that before we get to application of map $\frac{\tau}{2^m}$ for chosen $m$, our set spreads its vertical length $2^m$--times (or it reaches the full length first).\\

Now we take sequence $\{n_i\}$ satisfying (\ref{condition}), as desired. In that case the vertical length of set $\{x_0\} \times [0,\frac{1}{2^j}]$ will grow even faster. The same estimate can be done for any point $x_0 \in Q \setminus K$ (not necessarily starting from fibre with full tent map), which proves (1).\\

To prove (2) it is enough to realize that the vertical length of every $\{x_0\} \times [0,\frac{1}{2^j}]$ can maximally spread to full length 1 and periodically, maps $\frac{\tau}{2^k}$ for all $k=1,2,...$ are applied. Every such map shrinks the length of the set to $\frac{1}{2^{k-1}}$, which for $k \to \infty$ gives (2).
\end{proof}

\bigskip

\begin{theorem}
There is a triangular map $F: I^2 \to I^2$ with positive topological entropy, hence is DC2, but containing no DC2-point. 
\label{noDCpoint}
\end{theorem}
\bigskip

\begin{proof}
Suppose the class $\mathcal{F}$ from Example \ref{example}. We will try to find sequence $\{n_i\}$ with corresponding map $F$ and a DC2--point $x_{DC} = (x_0,y_0)$ for $F$. By the definition from \cite{1}, $x_{DC}$ is DC2--point, if for every $\varepsilon>0$ there are distributionally scrambled set $S_{x_{DC},\varepsilon}$ (of type DC2), integer $n(x_{DC}, \varepsilon)$ and envelope $A_{x_{DC} , \varepsilon} \supset S_{x_{DC},\varepsilon}$ such that 
\begin{equation}
A_{x_{DC}, \varepsilon} \subseteq F^{n(x_{DC},\varepsilon)}(A_{x_{DC}, \varepsilon}) \subset B(x_{DC},\varepsilon).
\label{inclusions}
\end{equation}
\medskip

Let us gather some facts about our example:
\begin{enumerate}
\item Coordinate $x_0 \in Q$, because $Q \times I$ is the support of topological entropy of $F$ (see \cite{10}, or \cite{7}).
\item If point $x \in Q$ lies in the both-sided $f-$orbit of a point where the fibre map equals zero, then $R(F,x,j)=0$ eventually. Therefore $x_0 \in Q \setminus K$.  
\item Possible DC2--scrambled sets are always contained in the single fibers $I_x$, because different points of $Q$ are contained in different $J_k^n$'s, which keep their distance from each other when iterated. 
\item On the other hand, the envelopes cannot be subsets of the single fibers, because every fiber gets arbitrarily close to itself by $f$-iterating, but never exactly on itself. 
\item If $y_0\neq 0$, $x_{DC}$ is not a DC2--point. Regular applications of maps $\frac{\tau}{2^i}$ for high $i$ shrink possible envelopes arbitrarily close to zero, which makes them not invariant. Therefore $y_0=0$.
\label{3}
\end{enumerate}
\medskip

Consequently $x_{DC}=(x_0,0)$, $x_0 \in Q \setminus K$, $S_{x_{DC}, \varepsilon} \subseteq I_{x_0}$ and $A_{x_{DC},\varepsilon} \subseteq (J_{k_0}^m \cap Q) \times [0 ,\varepsilon)$ for some $k_0$ and $m$. For every $\varepsilon>0$ there is unique $m \in \mathbb{N}$ such that $\frac{1}{2^m} \leq \varepsilon < \frac{1}{2^{m-1}}$ and unique $J_{k_0}^m$ which contains $x_0$. Clearly $A_{x_{DC},\varepsilon}$ can be chosen closed and such that $S_{x_{DC},\varepsilon} \subset A_{x_{DC},\varepsilon}$. Envelope $A_{x_{DC},\varepsilon}$ will be "jagged" from above. Indeed, fibers in $(J_{k_0}^m \cap Q \cap K) \times [0 ,\varepsilon)$ can only consists of points $(J_{k_0}^m \cap Q \cap K) \times \{0\}$, otherwise the envelope would not be invariant. \\

As $J_{k_0}^m$ "travels" through different $J_k^m$'s, as described in (\ref{map g}), different fiber maps $\frac{\tau}{2^{n_i}}$ are applied. Since there are $2^m$ portions $J_k^m$ and infinitely many maps $\frac{\tau}{2^{n_i}}$ acting on them (sequence $\{n_i\}$ is increasing), there is one set $J_{p}^m$ and number $P \in \mathbb{N}$, such that all maps $\frac{\tau}{2^{n_j}}$, $j>P$ are acting in points of $J_p^m$. Obviously the higher $j$, the smaller subportion of $J_p^m$ where map $\frac{\tau}{2^{n_j}}$ is acting on.\\

The portion $J_{k_0}^m$ is $f$-periodic with period $2^{m}$, and by (\ref{form4}), $J_{k_0}^m \cap Q$ is also $f$-periodic with period $2^{m}$. Also any subportion $J_{k_1}^{m+i} \subset J_{k_0}^m$ for arbitrary $i \in \mathbb{N}$ is periodic with period $2^{m+i}$. \\

To make the $y$-axis of $A_{x_{DC}, \varepsilon} \subset (J_{k_0}^m \cap Q) \times [0 ,\frac{1}{2^m})$ also shrink back to $B(x_{DC}, \varepsilon)$ with period of suitable power of $2$ it is enough to choose $x_0$ from such subportion $J_{k_1}^{m+i} \subset f(J_p^m)$ that in points of $f^{-1}(J_{k_1}^{m+i})$ only maps $\frac{\tau}{2^{m+1}}$ or lower are applied (see (2) from Lemma \ref{helpful}). In such case the vertical length of fibres of $A_{x_{DC}, \varepsilon}$ will sufficiently shrink at iteration $2^m$ and the chosen subportion $J_{k_1}^{m+i}$ is periodic with period $2^{m+i}$. Consequently $F^{2^{m+i}} (A_{x_{DC}, \varepsilon}) \subset B(x_{DC}, \varepsilon)$. \\

On the other hand it is impossible to satisfy the first inclusion in (\ref{inclusions}). Exactly once in every $2^m$ iterations, $J_{k_0}^m$ maps to $J_p^m$, where all fiber maps $\frac{\tau}{2^{n_j}}, j>P$ are applied. Therefore for any $n \in \mathbb{N}$ always there is a part of $A_{x_{DC},\varepsilon}$ such that the fibers in $F^n(A_{x_{DC},\varepsilon})$ are shorter than the original ones in $A_{x_{DC},\varepsilon}$.
\end{proof}

\bigskip

 {\bf Acknowledgment.}
The second author thanks Professor Jaroslav Sm\'{\i}tal for his valuable suggestions.

\bigskip

\end{document}